\documentclass[reqno]{amsart}
\usepackage{amsmath,amsthm,amssymb,bm}

\makeatletter
    
    \@addtoreset{equation}{section}
\makeatother

\newtheorem{Thm}{Theorem}[section]
\newtheorem{Lem}[Thm]{Lemma}
\newtheorem{Cor}[Thm]{Corollary}
\newtheorem{Prop}[Thm]{Proposition}
\newtheorem{Rem}[Thm]{Remark}

\newcommand{\C}{\mathbb{C}}           
\newcommand{\Z}{\mathbb{Z}}

\newcommand{\ch}{\mathrm{ch}\,}

\newcommand{\fb}{{\mathfrak b}}
\newcommand{\fg}{{\mathfrak g}}
\newcommand{\fh}{{\mathfrak h}}

\newcommand{\fn}{{\mathfrak n}}
\newcommand{\fp}{{\mathfrak p}}

\newcommand{\hfg}{\hat{\fg}}

\newcommand{\hfb}{\hat{\fb}}

\newcommand{\hfn}{\hat{\fn}}
\newcommand{\hfp}{\hat{\fp}}

\newcommand{\hP}{\hat{P}}
\newcommand{\hQ}{\hat{Q}}
\newcommand{\hI}{\hat{I}}
\newcommand{\hW}{\hat{W}}

\newcommand{\ga}{\alpha}
\newcommand{\gb}{\beta}

\newcommand{\gl}{\lambda}
\newcommand{\gL}{\Lambda}
\newcommand{\gd}{\delta}
\newcommand{\gD}{\Delta}

\newcommand{\go}{\omega}

\renewcommand{\ggg}{\gamma}

\renewcommand{\hat}{\widehat}
\newcommand{\ol}{\overline}

\newcommand{\wt}{\mathrm{wt}}

\newcommand{\ba}{\bm{a}}
\renewcommand{\bf}{\bm{f}}

\newcommand{\be}{\bm{e}}

\newcommand{\br}{\bm{r}}
\newcommand{\bs}{\bm{s}}
\newcommand{\pr}{\mathrm{pr}}
\newcommand{\bL}{\mathbf{L}}
\newcommand{\hV}{\hat{V}}
\newcommand{\hgDre}{\hat{\gD}^{\mathrm{re}}}

\begin{document}

\title[Graded limits of minimal affinizations of type $G_2$]{Graded limits of minimal affinizations over the quantum loop algebra of type $G_2$}
\author{Jian-Rong Li and Katsuyuki Naoi}
\address{Jian-Rong Li, School of Mathematics and Statistics, Lanzhou University, Lanzhou 730000, P. R. China, and Einstein Institute of Mathematics, The Hebrew University of Jerusalem, Jerusalem 9190401, Israel}
\email{lijr07@gmail.com, lijr@lzu.edu.cn}
\address{Katsuyuki Naoi: Institute of Engineering, Tokyo University of Agriculture and Technology, 3-8-1 Harumi-cho, Fuchu-shi, Tokyo, Japan}
\email{naoik@cc.tuat.ac.jp}
\date{}
\maketitle

\begin{abstract}
The aim of this paper is to study the graded limits of minimal affinizations over the quantum loop algebra of type $G_2$. 
We show that the graded limits are isomorphic to multiple generalizations of Demazure modules, and obtain defining
relations of them. As an application, we obtain a polyhedral multiplicity formula for the decomposition of minimal affinizations of type $G_2$
as a $U_q(\mathfrak{g})$-module, by showing the corresponding formula for the graded limits. As another application, we prove a character formula of the least affinizations of generic parabolic Verma modules of type $G_2$ conjectured by Mukhin and Young.

\hspace{0.15cm}


\hspace{0.15cm}

\end{abstract}

\section{Introduction}
Let $\mathfrak{g}$ be a simple Lie algebra, $\mathbf{L}\mathfrak{g} = \mathfrak{g}\otimes \mathbb{C}[t, t^{-1}]$ the 
corresponding loop algebra, and $U_q(\mathbf{L}\mathfrak{g})$ the corresponding quantum loop algebra. 
Minimal affinizations of representations of quantum groups are an important family of simple 
$U_q(\mathbf{L}\mathfrak{g})$-modules which was introduced in \cite{MR1367675}. 
The celebrated Kirillov-Reshetikhin modules are examples of minimal affinizations.

Graded limits of minimal affinizations, which are graded analogs of the classical limits defined over the current algebra
$\fg[t] = \fg \otimes \C[t]$, were studied in \cite{MR1836791}, \cite{MR2238884}, \cite{MR2587436}, \cite{MR2896463},
\cite{MR3120578}, \cite{MR3210588}. 

Minimal affinizations over the quantum loop algebra of type $G_2$ were studied in \cite{MR1367675}, \cite{MR2372556}, \cite{MPr07}, \cite{LM13}, \cite{QL14}. 
The aim of this paper is to study the graded limits of minimal affinizations over the quantum loop algebra of type $G_2$. 

Assume that $\fg$ is of type $G_2$.
Let $L(m)$ be the graded limit of a minimal affinization with highest weight $\gl$, and let $M(\gl)$ be the $\fg[t]$-module 
generated by a nonzero vector $v_\gl$ with certain relations. 
Our first main result (Theorem \ref{isomorphism of modules}) is that $M(\gl) \cong L(m) \cong T(\gl)$, where $T(\gl)$ is 
some generalized Demazure module. 
These isomorphisms were previously conjectured by Moura in \cite{MR2587436}.

Let $\go_1$ (resp.\ $\go_2$) be the fundamental weight with respect to the long (resp.\ short) simple root, 
and assume that $\gl = k\go_1+l\go_2$.
Using the above isomorphisms, we obtain the following polyhedral multiplicity formula as a $\fg$-module
(Theorem \ref{polyhedral multiplicity formula})
\[ L(m) \cong \bigoplus_{(a_1,\ldots,a_5) \in S_\gl} V\big((k -a_1 + a_3 + a_4 -a_5)\go_1+(l-a_2 -3a_3-3a_4)\go_2\big)
\]
where
\[ S_\gl=\big\{(a_1,\ldots,a_5) \in \Z_+^5 \bigm| a_1 \leq k, \ a_1 - a_3 + a_5 \leq k, \ 2a_2 + 3a_3 + 3a_4 \leq l,\ 
   2a_2+3a_4+3a_5 \leq l\big\}.
\]
Here $V(\mu)$ denotes the simple $\fg$-module with highest weight $\mu$.
As an immediate corollary, we obtain a similar formula for the multiplicity of minimal affinizations as a $U_q(\fg)$-module (Corollary \ref{multiplicity formula for decomposition of minimal affinizations of type G2}). This formula is a generalization of the one given in \cite{MR2372556}, in which the formula for Kirillov-Reshetikhin modules 
(i.e.\ the case $k=0$ or $l=0$) is given.

We also give a formula for the limit of normalized characters (Corollary \ref{Conjecture of the character formula of least 
affinizations of Verma modules by Mukhin and Young}), which yields the character formula of least affinizations of 
generic parabolic Verma modules of type $G_2$ 
conjectured by Mukhin and Young \cite[Conjecture 6.3]{MY14}.

The paper is organized as follows. In Section \ref{Section:Background}, 
we give some background information about the quantum loop algebra of type $G_2$. 
In Section \ref{Section:Main}, we describe our main results in this paper. 
In Section \ref{Section:isom}, we prove Theorem \ref{isomorphism of modules}. In Section \ref{Section:Polyhedral}, 
we prove Theorem \ref{polyhedral multiplicity formula}.

\section{Background}\label{Section:Background}

Let $\Z$ be the set of integers, and $\Z_+$ the set of nonnegative integers.
In this paper, we take $\mathfrak{g}$ to be the complex simple Lie algebra of type $G_2$. 
Let $\fh$ be a Cartan subalgebra and $\fb$ a Borel subalgebra containing $\fh$.
Let $I=\{1, 2\}$.
We choose simple roots $\alpha_1, \alpha_2$ and scalar product $(\cdot, \cdot)$ such that
\begin{align*}
( \alpha_1, \alpha_1 ) = 6, \ ( \alpha_1, \alpha_2 )=-3, \ ( \alpha_2, \alpha_2 )=2.
\end{align*}
Therefore $\alpha_1$ is the long simple root and $\alpha_2$ is the short simple root. The set of long positive roots is
\begin{align*}
\{ \alpha_1, \ \alpha_1 + 3 \alpha_2, \ 2 \alpha_1 + 3 \alpha_2 \}.
\end{align*}
The set of short positive roots is
\begin{align*}
\{ \alpha_2, \ \alpha_1 + \alpha_2, \ \alpha_1 + 2 \alpha_2 \}.
\end{align*}
Let $\Delta_+$ denote the set of positive roots. Denote by $\gD$ the root system of $\fg$, and by $\gD_+$ the set of positive roots.
Let $W$ denote the Weyl group with simple reflections $s_i$ ($i\in I$).
Denote by $\fg_\ga$ ($\ga \in \gD$) the corresponding root space,
and for each $\ga \in \gD_+$ fix nonzero elements $e_\ga \in \fg_\ga$, $f_\ga \in \fg_{-\ga}$ and $\ga^\vee \in \fh$ such that
\[ [e_\ga, f_\ga] = \ga^\vee, \ \ \ [\ga^\vee, e_\ga] = 2 e_\ga, \ \ \ [\ga^\vee, f_\ga] = -2 f_\ga.
\]
We also use the notation $e_i = e_{\ga_i}$, $f_i = f_{\ga_i}$ for $i \in I$, and $e_{-\ga} = f_\ga$ for $\ga \in \gD_+$.
Set $\fn_{\pm} = \bigoplus_{\ga \in \gD_+} \fg_{\pm \ga}$.

Let $\omega_i$ ($i \in I$) be the fundamental weight.
We have $\omega_1 = 2 \alpha_1 + 3 \alpha_2$, $\omega_2 = \alpha_1 + 2 \alpha_2$.
Let $P$ be the weight lattice, and
\[ P_+ = \sum_{i\in I}\Z_+ \go_i \subseteq P, \ \ \ Q_+ = \sum_{i \in I} \Z_+ \ga_i \subseteq P.
\]
Note that $P$ coincides with the root lattice $\sum_{i\in I} \Z\ga_i$, but $P_+ \neq Q_+$.
We write $\gl \le \mu$ for $\gl, \mu \in P$ if $\mu-\gl \in Q_+$.
For $\gl \in P_+$, denote by $V(\gl)$ the simple $\fg$-module with highest weight $\gl$.

Let $\hfg = \fg \otimes \C[t,t^{-1}] \oplus \C K \oplus \C d$ be the affine Kac-Moody Lie algebra associated with $\fg$,
where $K$ is the canonical central element and $d$ is the degree operator.
Let $\hI = \{0,1,2\}$, and 
\[ e_0 = f_{2\ga_1+3\ga_2} \otimes t, \ \ \ f_0 = e_{2\ga_1+3\ga_2} \otimes t^{-1}.
\]
In this paper, we put \ $\hat{}$ \ to denote the objects associated with $\hfg$.
For example, $\hP$ and $\hQ$ denote the weight and root lattices of $\hfg$ respectively, and so on.
Let $\gd \in \hP$ be the null root, and denote by $\gL_0 \in \hP_+$ the unique dominant integral weight of $\hfg$ satisfying 
\[ \langle \ga_i^\vee,\gL_0\rangle = 0 \ \text{for} \ i\in I, \ \ \ \langle K,\gL_0\rangle = 1, \ \ \ \langle d,\gL_0 \rangle
   = 0.
\]

Let $\bL\fg = \fg \otimes \C[t,t^{-1}]$ and $\fg[t] = \fg \otimes \C[t]$ be the loop algebra and the current algebra 
associated with $\fg$ respectively, whose Lie algebra structures are given by
\[ [x \otimes f(t), y \otimes g(t)] = [x,y] \otimes f(t)g(t).
\]
Note that $\fg[t]$ is naturally considered as a Lie subalgebra of $\hfg$.

Quantum groups are introduced independently by Jimbo \cite{MR797001} and Drinfeld \cite{MR914215}. Quantum loop algebra are infinite-dimensional quantum groups. The quantum loop algebra $U_q(\bL\fg)$ in Drinfeld's new realization is a $\C(q)$-algebra 
generated by $x_{i, n}^{\pm}$ ($i\in I, n\in 
\mathbb{Z}$), $k_i^{\pm 1}$ $(i\in I)$, $h_{i, n}$ ($i\in I, n\in \mathbb{Z}\backslash \{0\}$), 
subject to certain relations, see \cite{MR914215}.
Denote by $U_q(\fg)$ the subalgebra of $U_q(\bL\fg)$ generated by $x_{i,0}^{\pm}$ ($i \in I$), $k_i^{\pm 1}$ ($i \in I$),
which is isomorphic to the quantized enveloping algebra associated with $\fg$.
For $\gl \in P_+$, let $V_q(\gl)$ denote the finite-dimensional simple $U_q(\fg)$-module of type $1$
with highest weight $\gl$.

Simple $U_q(\mathbf{L}\mathfrak{g})$-modules are parametrized by dominant monomials in 
$\mathbb{Z}[Y_{i,a}^{\pm 1}]_{i \in I, a \in \mathbb{C}(q)^{\times}}$, where $Y_{i,a}^{\pm 1}$'s are formal variables, and 
a monomial $m = \prod_{i \in I, a \in \mathbb{C}(q)^{\times}}Y_{i,a}^{u_{i,a}}$ is dominant if $u_{i,a} \geq 0$ for all 
$i$ and $a$ (see \cite{MR1357195}, or \cite{MR1745260} for the present formulation). 
For a dominant monomial $m$, denote by $L_q(m)$ the corresponding simple $U_q(\mathbf{L}\mathfrak{g})$-module. Let $\mathcal{P}_+$ be the monoid generated by $\{Y_{i, a}| i \in I, a \in \mathbb{C}^{\times} q^{\mathbb{Z}}\}$.

Let $\lambda = k\omega_1 + l \omega_2$, $k, l \in \mathbb{Z}_+$. A simple $U_q (\bL\fg) $-module $L_q(m)$ is a 
\textit{minimal affinization} of $V_q(\lambda)$ if and only if $m$ is one of the following monomials
\begin{align}\label{eq:l-highest-wt}
\left( \prod_{i=0}^{k-1} Y_{1, aq^{6i}} \right) \left( \prod_{i=0}^{l-1} Y_{2, aq^{6k+2i+1}} \right), \qquad
\left( \prod_{i=0}^{l-1} Y_{2, aq^{2i}} \right) \left( \prod_{i=0}^{k-1} Y_{1, aq^{2l+6i+5}} \right),
\end{align}
for some $a\in \mathbb{C}(q)^{\times}$, see \cite{MR1347873}.

\section{Main results}\label{Section:Main}

The aim of this paper is to study the graded limits of minimal affinizations in type $G_2$.
So let us recall the definition of the graded limits.

Let $\gl = k\go_1 + l \go_2$, and $m$ be one of the monomials in (\ref{eq:l-highest-wt}).
Without loss of generality, we may assume that $a \in \C^\times$.
Let $\mathbf{A} = \C[q,q^{-1}]$, $U_{\mathbf{A}}(\bL\fg)$ be the $\mathbf{A}$-lattice of $U_q(\bL\fg)$ 
(see \cite{MR1300632}), and $L_{\mathbf{A}}(m) = U_{\mathbf{A}}(\bL\fg)v_m$ where $v_m$ is a highest $\ell$-weight vector of 
$L_q(m)$.
Then 
\[ \ol{L_q(m)} = L_{\mathbf{A}}(m)\otimes_{\mathbf{A}} \C
\]
becomes a finite-dimensional $\bL\fg$-module called the \textit{classical limit} of $L_q(m)$, 
where we identify $\C$ with $\mathbf{A}/\langle q -1 \rangle$.
Define a Lie algebra automorphism $\varphi_a\colon \fg[t] \to \fg[t]$ by 
\[ \varphi_a\big(x \otimes f(t)\big) = x \otimes f(t - a) \ \ \ \text{for} \ x \in \fg, f \in \C[t].
\]
Now we consider $\ol{L_q(m)}$ as a $\fg[t]$-module by restriction, and
define a $\fg[t]$-module $L(m)$ by the pull-back $\varphi_a^*\big(\ol{L_q(m)}\big)$.
We call $L(m)$ the \textit{graded limit} of $L_q(m)$. 
By the construction we have for every $\mu \in P_+$ that 
\begin{equation}\label{eq:multi}
 \Big[L_q(m) : V_q(\mu)\Big] = \Big[L(m): V(\mu)\Big],
\end{equation}
where the left- and right-hand sides are the multiplicities 
as a $U_q(\fg)$-module and $\fg$-module, respectively.

Now we shall state our first main theorem, which gives isomorphisms between $L(m)$ and other two 
$\fg[t]$-modules.
Let $M(\gl)$ be the $\fg[t]$-module generated by a nonzero vector $v_M$ with relations
\begin{align}\label{eq:relations_of_M}
 \fn_+[t]v_M=0, \ \ \ &(h \otimes t^k)v_M = \gd_{k0}\langle h,\gl\rangle v_M \ \mathrm{for} \ h \in \fh, \ \ \ 
 f_i^{\langle \ga_i^\vee, \gl \rangle +1}v_M=0 \ \text{for} \ i \in I,\nonumber\\
 &(f_{\ga_1} \otimes t) v_M =0, \ \ \ (f_{\ga_2} \otimes t) v_M = 0,\ \ \ (f_{\ga_1+\ga_2}\otimes t)v_M=0.
\end{align}
The other $\fg[t]$-module is a multiple generalization of a Demazure module defined as follows.
Let $\xi_1,\ldots,\xi_p$ be a sequence of elements of $\hat{P}$, and assume for each $1\le i \le p$ that there exists $\gL^i 
\in \hP_+$ such that $\xi_i$ belongs to the affine Weyl group orbit $\hW\gL^i$ of $\gL^i$.
Let $\hV(\gL^i)$ denote the simple highest weight $\hfg$-module with highest weight $\gL^i$, and $v_{\xi_i} \in 
\hV(\gL^i)_{\xi_i}$ be an extremal weight vector with weight $\xi_i$.
We define a $\hfb$-module $D(\xi_1,\ldots,\xi_p)$ by 
\begin{equation}\label{eq:module_D}
 D(\xi_1,\ldots,\xi_p) = U(\hfb)(v_{\xi_1}\otimes \cdots \otimes v_{\xi_p}) \subseteq \hV(\gL^1) \otimes \cdots \otimes 
 \hV(\gL^p).
\end{equation}
Here $\hfb = \fb \oplus \C K \oplus \C d \oplus t\fg[t]$ is the standard Borel subalgebra of $\hfg$.

\begin{Rem}\normalfont
 For any $c_1,\ldots,c_p \in \Z$, it obviously holds that
 \[ D(\xi_1 + c_1\gd, \ldots,\xi_p + c_p\gd) \cong D(\xi_1,\ldots,\xi_p)
 \]
 as $\big(\fb \oplus t\fg[t]\big)$-modules.
\end{Rem}

Now write $l = 3r + s$ with $r \in \Z_+$, $s \in \{0,1,2\}$, and set
\[ T(\gl) = \begin{cases} D\big(k(-\go_1+\gL_0), r(-3\go_2 + \gL_0)\big) & \text{if} \ s = 0,\\
                          D\big(k(-\go_1+\gL_0),r(-3\go_2+\gL_0), -s\go_2+\gL_0\big) & \text{otherwise}.
            \end{cases}
\]
Note that $T(\gl)$ is extended to a module over $\fg[t] \oplus \C K \oplus \C d$,
and as a $\fg[t]$-module $T(\gl)$ is generated by the one-dimensional weight space $T(\gl)_\gl$.

Our first main theorem is the following theorem.
\begin{Thm} \label{isomorphism of modules}
 As a $\fg[t]$-module, we have
 \[  M(\gl) \cong L(m)\cong T(\gl).
 \]
\end{Thm}

The second main theorem gives a multiplicity formula for $L(m)$ as a $\fg$-module.
For $\gl = k\go_1+l\go_2$, define a subset $S_\gl \subseteq \Z_+^5$ by 
\[ S_\gl = \big\{ (a_1,\ldots,a_5)\bigm| a_1 \leq k, \ a_1 - a_3 + a_5 \leq k, \ 2a_2 + 3a_3 + 3a_4 \leq l,\ 2a_2+3a_4+3a_5 
   \leq l\big\}.
\]

\begin{Thm} \label{polyhedral multiplicity formula}
 As a $\fg$-module, 
 \[ L(m) \cong \bigoplus_{(a_1,\ldots,a_5) \in S_\gl} V\big((k -a_1 + a_3 + a_4 -a_5)\go_1+(l-a_2 -3a_3-3a_4)\go_2\big).
 \]
\end{Thm}

By (\ref{eq:multi}), we immediately obtain the following corollary.

\begin{Cor} \label{multiplicity formula for decomposition of minimal affinizations of type G2}
 As a $U_q(\fg)$-module,
 \[ L_q(m)\cong \bigoplus_{(a_1,\ldots,a_5) \in S_\gl} V_q\big((k -a_1 + a_3 + a_4 -a_5)\go_1+(l-a_2 -3a_3-3a_4)\go_2\big).
 \]
\end{Cor}

From Theorem 3.2, we also obtain the following formula for the limit of the (normalized) characters of minimal affinizations.

\begin{Cor}\label{Conjecture of the character formula of least affinizations of Verma modules by Mukhin and Young}
Let $J$ be a subset of $I$, and suppose that $\lambda_1, \lambda_2, \ldots$ is an infinite sequence of elements of 
$P_+$ such that 
\[ \lim_{n \to \infty} \langle \alpha_i^{\vee}, \lambda_n \rangle = \infty \text{ for all } i \in J \text{ and } 
   \langle\alpha_i^{\vee}, \lambda_n \rangle = 0 \text{ for all } i \notin J, \ n \in \mathbb{Z}_{>0}.
\] 
Let $m_1, m_2, \ldots$ be an infinite sequence of elements of $\mathcal{P}_+$ such that $L_q(m_n)$ is a minimal affinization of $V_q(\lambda_n)$. Then $\lim_{n \to \infty} e^{-\lambda_n} \ch L_q (m_n)$ exists, and
\begin{align*}
\lim_{n \to \infty} e^{-\lambda_n} \ch L_q (m_n) = \prod_{\alpha \in \Delta_+} \left( \frac{1}{1-e^{-\alpha}} 
\right)^{\max_{j \in J} \langle \omega_j^{\vee}, \alpha \rangle}.
\end{align*}
\end{Cor}

\begin{proof}
This result follows from Theorem \ref{isomorphism of modules}, and the proof is the same as the proof of 
\cite[Corollary 4.13]{MR3120578}.
\end{proof}

This corollary, together with \cite[Corollary 5.6]{MY14}, yields the character formula of the least affinizations of 
generic parabolic Verma modules of type $G_2$ conjectured by Mukhin and Young \cite[Conjecture 6.3]{MY14}.

\section{Proof of Theorem \ref{isomorphism of modules}}\label{Section:isom}

Throughout the rest of this paper, we fix $\gl = k\go_1 + l\go_2 \in P_+$ and set $r\in \Z_+$ and
$s \in \{0,1,2\}$ to be such that $l = 3r+s$.
Let $m$ be one of the monomials in (\ref{eq:l-highest-wt}).
In this section, we shall prove one by one the existence of three surjective homomorphisms
\[ M(\gl) \twoheadrightarrow L(m), \ \ \ L(m) \twoheadrightarrow T(\gl), \ \ \ T(\gl) \twoheadrightarrow M(\gl),
\]
which completes the proof of Theorem \ref{isomorphism of modules}.

\subsection{Proof of {\boldmath $M(\gl) \twoheadrightarrow L(m)$}}\label{subsection:MtoL}

Let $v_m$ be a highest $\ell$-weight vector of $L_q(m)$,
and $W = U_q(\fg)v_m \subseteq L_q(m)$ the simple $U_q(\fg)$-submodule generated by $v_m$.
It follows from \cite[Proposition 5.5]{MR1367675} that $\bigoplus_{\mu \geq \gl-\ga_1-\ga_2} L_q(m)_{\mu} \subseteq W$, 
where $L_q(m)_\mu$ denotes the weight space with weight $\mu$.
Hence we have
\[ x_{\ga_1,1}^-v_m \in W,\ \ x_{\ga_2,1}^-v_m \in W, \ \ [x_{\ga_1,1}^-,x_{\ga_2,0}^-]v_m \in W.
\]
Then it is proved from the definition of the graded limit that the vector 
$\bar{v}_m = v_m\otimes_{\mathbf{A}} 1 \in L(m)$ satisfies
\[ (f_{\ga_1}\otimes t)\bar{v}_m= (f_{\ga_2}\otimes t)\bar{v}_m = (f_{\ga_1+\ga_2}\otimes t)\bar{v}_m=0
\]
(see \cite[Subsection 4.1]{MR3120578}).
The other relations in (\ref{eq:relations_of_M}) are easily checked from the construction.
Hence $M(\gl) \twoheadrightarrow L(m)$ follows.

\subsection{Proof of {\boldmath $L(m) \twoheadrightarrow T(\gl)$}}

Here we only consider the case where the monomial $m$ is of the form $\prod_{i=0}^{k-1} Y_{1, aq^{6i}}\cdot 
\prod_{i=0}^{l-1} Y_{2, aq^{6k+2i+1}}$. The proof of the other case is similar.

Set 
\[ m_1 = \prod_{i=0}^{k-1} Y_{1, aq^{6i}}, \ \ \ m_2 = \prod_{i=0}^{l-1} Y_{2, aq^{6k+2i+1}}.
\]
By \cite[Theorem 5.1]{MR1883181} (or more precisely, the dualized statement of it), there exists an injective 
homomorphism
\[ L_q(m) \hookrightarrow L_q(m_1) \otimes L_q(m_2)
\]
mapping a highest $\ell$-weight vector to the tensor product of highest $\ell$-weight vectors.
Then by the definition of graded limits, we obtain a $\fg[t]$-module homomorphism
\begin{equation*}\label{eq:map}
 L(m) \to L(m_1) \otimes L(m_2)
\end{equation*}
mapping a highest weight vector to the tensor product of highest weight vectors.
Now the existence of a surjection $L(m) \twoheadrightarrow T(\gl)$ is proved from 
the following lemma.

\begin{Lem}
 {\normalfont(i)} $L(m_1)$ is isomorphic to $D\big(k(-\go_1+\gL_0)\big)$ as a $\fg[t]$-module.\\
 {\normalfont(ii)} $L(m_2)$ is isomorphic to $D\big(r(-3\go_2+\gL_0)\big)$ {\normalfont(}resp.\ $D\big(r(-3\go_2+\gL_0),
  -s\go_2 + \gL_0\big)${\normalfont)} if $s= 0$ {\normalfont(}resp.\ $s=1,2${\normalfont)} as a $\fg[t]$-module.
\end{Lem}

\begin{proof}
 The graded limit $L(m_1)$ is isomorphic to the Kirillov-Reshetikhin module $KR(k\go_1)$ for $\fg[t]$ defined in 
 \cite{MR2238884,MR2372556}, which is proved from the facts that there exists a surjection $KR(k\go_1) \twoheadrightarrow 
 L(m_1)$ (see Subsection \ref{subsection:MtoL}) and the characters of two modules are the same 
 (see \cite{MR1745263,MR2254805,MR2372556}). Hence the assertion (i) follows from \cite[Theorem 4]{MR2323538}.
 Similarly $L(m_2)$ is isomorphic to $KR(l\go_2)$, and hence by 
 \cite[Corollary 2.3]{MR2372556} it is isomorphic to the $\fg[t]$-submodule of
 $KR(3r\go_2) \otimes KR(s\go_2)$ generated by the tensor product of highest weight vectors.
 Now $KR(3r\go_2) \cong D\big(r(-3\go_2+\gL_0)\big)$ follows from \cite[Theorem 4]{MR2323538}, and 
 $KR(s\go_2) \cong D(-s\go_2+\gL_0)$ is verified by the Demazure character formula (see \cite{MR2323538}).
 Hence the assertion (ii) is proved.
\end{proof}

\subsection{Proof of {\boldmath$T(\gl) \twoheadrightarrow M(\gl)$}}

First we introduce the following notation, as in \cite{MR3120578,MR3210588}.
Assume that $V$ is a $\hfg$-module and $D$ is a $\hfb$-submodule of $V$.
For $i \in \hI$ let $\hfp_i$ denote the parabolic subalgebra $\hfb \oplus \C f_i \subseteq \hfg$,
and set $F_iD= U(\hfp_i)D \subseteq V$ to be the $\hfp_i$-submodule generated by $D$.
It is easily seen that, if $\xi_1,\ldots,\xi_p \in \hW(\hP_+)$ satisfy $\langle \ga_i^{\vee}, \xi_j\rangle \ge 0$ for all 
$1\le j \le p$, then
\begin{equation}\label{eq:recursive}
 F_iD(\xi_1,\ldots,\xi_p) = D(s_i\xi_1,\ldots,s_i\xi_p)
\end{equation}
(see \cite[Lemma 2.4]{MR3120578}).

Let $\hgDre = \gD + \Z\gd$ be the set of real roots of $\hfg$,
and $\hgDre_+ = \gD_+ \sqcup (\gD + \Z_{> 0}\gd)$ the set of positive real roots.
For $\ggg = \ga+p\gd \in \hgDre$, set
\[ \ggg^{\vee} = \ga^{\vee} + \frac{6p}{(\ga,\ga)}K,
\]
and define a number $\rho(\ggg)$ by
\[ \rho(\ggg)=\max\!\big\{0, -\langle \ggg^{\vee}, k(\go_1+\gL_0)\rangle\big\} + \max\!\big\{0, -\langle \ggg^{\vee}, 
   r(3\go_2+\gL_0)\rangle\big\} + \max\!\big\{0,-\langle \ggg^{\vee},s\go_2+\gL_0\rangle\big\}.
\]
The explicit values of $\rho(\ggg)$ for $\ggg \in \hgDre_+$ are given as follows:
\begin{align*}
 \rho\big(-(\ga_1+2\ga_2)+ \gd\big) &= 3r+\gd_{s2}, \\
 \rho\big(-(\ga_1+3\ga_2)+\gd\big) &= 2r + \gd_{s2},\\
 \rho\big(-(2\ga_1+3\ga_2)+\gd\big) &=k+2r + \gd_{s2},\\
 \rho\big(-(\ga_1+3\ga_2)+2\gd\big) &=\rho\big(-(2\ga_1+3\ga_2)+2\gd)=r,
\end{align*}
and $\rho(\ggg) = 0$ for all the other $\ggg \in \hgDre_+$.
Here $\gd_{s2}$ denotes the Kronecker's delta.
For $\ga + p\gd \in \hgDre$ set $x_{\ga + p\gd} = e_\ga \otimes t^p$.

Recall that $v_\xi$ denotes an extremal weight vector in $\hV(\gL)$ with weight $\xi$, where
$\gL\in \hP_+$ is the element satisfying $\xi \in \hW\gL$. 
Let $v_T \in T(\gl)$ be the tensor product of the extremal weight vectors:
\[ v_T = \begin{cases}
          v_{k(\go_1+\gL_0)} \otimes v_{r(3\go_2 + \gL_0)} & s=0, \\
          v_{k(\go_1+\gL_0)} \otimes v_{r(3\go_2 + \gL_0)} \otimes v_{s\go_2+\gL_0} & s=1,2.
         \end{cases}
\]
Note that $T(\gl)$ is generated by $v_T$ as a $\fg[t]$-module.
Throughout the rest of this paper, we will abbreviate $X\otimes t^p$ as $X t^p$ to shorten the notation.

\begin{Lem}\label{Lem:Ann}
 We have
 \[ \mathrm{Ann}_{U(\hfn_+)}(v_T) = U(\hfn_+)\Big(\bigoplus_{\ggg \in \hat{\gD}_+^{\mathrm{re}}} \C x_\ggg^{\rho(\ggg)+1} + 
    \C f_{\ga_1+3\ga_2}t^2(f_{\ga_1+2\ga_2}t)^{3r-2}+ t\fh[t]\Big),
 \]
 where $\C f_{\ga_1+3\ga_2}t^2(f_{\ga_1+2\ga_2}t)^{3r-2}$ is omitted if $r=0$.
\end{Lem}

\begin{proof}
 First assume that $s=0$, and set $\gL = r(-2\go_1+3\go_2+\gL_0)$.
 Note that 
 \[ F_0D(k\gL_0, \gL) \cong D\big(k(\go_1+\gL_0), r(3\go_2+\gL_0)\big) \big(= U(\hfb)v_T\big)
 \]
 holds by (\ref{eq:recursive}), and we have
 \[ \mathrm{Ann}_{U(\hfn_+)}(v_{k\gL_0}\otimes v_{\gL}) =\mathrm{Ann}_{U(\hfn_+)}(v_\gL)
 \]
 since $\hfn_+$ acts trivially on $v_{k\gL_0}$. 
 We shall check that $D(k\gL_0, \gL)$ satisfies the conditions (i) -- (iii) (for $T$) in \cite[Lemma 5.3]{MR3120578}.
 Note that the condition (iii) holds by \cite[Theorem 5]{MR1887117}. 
 By \cite[Lemma 26]{MR932325}, we have
 \begin{align*}
  \mathrm{Ann}_{U(\hfn_+)}(v_\gL) &= U(\hfn_+)\Big(\bigoplus_{\ggg \in \hat{\gD}_+^{\mathrm{re}}} 
    \C x_{\ggg}^{\max\{0,-\gL(\ggg^{\vee})\}+1} + t\fh[t]\Big)\\
   &= U(\hfn_+)e_0 + U(\hfn_+) \Big(\bigoplus_{\ggg \in \hgDre_+ \setminus \{\ga_0\}} 
    \C x_{\ggg}^{\max\{0,-\gL(\ggg^{\vee})\}+1} + t\fh[t]\Big).
 \end{align*}
 It follows that
 \[ \max\{0,-\gL(\ggg^{\vee})\} = \begin{cases} 3r & \ggg= \ga_1+\ga_2,\\
                                                2r & \ggg= \ga_1,\\
                                                r  & \ggg= \ga_1+\gd \ \text{or} \ 2\ga_1+3\ga_2,\\
                                                0  & \text{otherwise.}
                                  \end{cases}
 \]
 Let $\hfn_0$ be the Lie subalgebra $\bigoplus_{\ggg \in \hgDre_+\setminus\{\ga_0\}} \C x_{\ggg} \oplus t\fh[t]$ of $\hfn_+$,
 and define a left $U(\hfn_0)$-ideal $\mathcal{I}$ by
 \[ \mathcal{I} = U(\hfn_0) \Big(\bigoplus_{\ggg \in \hgDre_+ \setminus \{\ga_0\}}
    \C x_{\ggg}^{\max\{0,-\gL(\ggg^{\vee})\}+1} +\C e_{\ga_1}te_{\ga_1+\ga_2}^{3r-2} + t\fh[t]\Big).
 \]
 It is directly checked for every $p \in \Z_+$ that
 \[ \mathrm{ad}(e_0)(e_{\ga_1+\ga_2}^{p}) \in \C^\times e_{\ga_1+\ga_2}^{p-1}f_{\ga_1+2\ga_2}t + \C^\times 
    e_{\ga_1+\ga_2}^{p-2}f_{\ga_2}t + \C^\times e_{\ga_1+\ga_2}^{p-3}e_{\ga_1}t,
 \]
 where we set $e_{\ga_1+\ga_2}^q=0$ if $q<0$.
 Using this we see that $\mathcal{I}$ is $\mathrm{ad}(e_0)$-invariant, and 
 \[ \mathrm{Ann}_{U(\hfn_+)}(v_\gL)= U(\hfn_+)e_0 + U(\hfn_+)\mathcal{I}.
 \]
 Now the assertion (for $s=0$) follows by \cite[Lemma 5.3]{MR3120578}.
 
 The case $s=1$ is easily proved from the case $s=0$ since $\hfn_+$ acts trivially on $v_{\go_2+\gL_0}$ and hence
 \[ \mathrm{Ann}_{U(\hfn_+)}\big(v_{k(\go_1+\gL_0)} \otimes v_{r(3\go_2+\gL_0)} \otimes v_{\go_2+\gL_0}\big) 
    = \mathrm{Ann}_{U(\hfn_+)}(v_{k(\go_1+\gL_0)} \otimes v_{r(3\go_2+\gL_0)}\big).
 \]

 For the case $s=2$, notice by (\ref{eq:recursive}) that
 \begin{align*}
  D\big(r(3\go_2+\gL_0), 2\go_2+\gL_0\big) \cong F_0F_1F_2F_1F_0D(r\gL_0, \go_2+\gL_0).
 \end{align*}
 Then this is isomorphic to
 \[ F_0F_1F_2F_1F_0D\big(\go_2+(r+1)\gL_0\big) \cong D\big((3r+2)\go_2+(r+1)\gL_0\big)
 \]
 since the $\hfg$-submodule of $\hV(r\gL_0) \otimes \hV(\go_2+\gL_0)$ generated by the tensor product of 
 highest weight vectors is isomorphic to $\hV\big(\go_2+(r+1)\gL_0\big)$.
 Hence we have 
 \[ D\big(k(\go_1+\gL_0), r(3\go_2+\gL_0), 2\go_2+\gL_0\big) \cong D\big(k(\go_1+\gL_0),(3r+2)\go_2+(r+1)\gL_0\big).
 \]
 Using this isomorphism, the assertion for $s=2$ is proved in almost the same way with the case $s=0$.
\end{proof}

Now Lemma \ref{Lem:Ann} and the following proposition yield a $(\fh \oplus \hfn_+)$-module homomorphism
from $U(\fh \oplus \hfn_+)v_T$ to $M(\gl)$ sending $v_T$ to $v_M$ since their weights are both $\gl$, 
and then the existence of a surjection
$T(\gl) \twoheadrightarrow M(\gl)$ is proved by the same argument with \cite[two paragraphs below Lemma 5.2]{MR3120578}.

\begin{Prop}\label{Prop:annihilation}
 The vector $v_M \in M(\gl)$ satisfies the relations
 \[ x_{\ggg}^{\rho(\ggg)+1}v_M=0 \ \text{for} \ \ggg \in \hgDre_+, \ \ \ t\fh[t]v_M = 0, \ \ \ f_{\ga_1+3\ga_2}t^2
    (f_{\ga_1+2\ga_2}t)^{3r-2}v_M=0,
 \]
 where the last one is omitted when $r = 0$.
\end{Prop}

The rest of this subsection is devoted to prove Proposition \ref{Prop:annihilation}.
For simplicity \textit{we assume that $s=0$ in the rest of this subsection}, and prove the proposition only in this case. 
The proof of the other cases are almost the same.
Note that the relations $x_{\ggg}v_M = 0$ for
\[ \ggg \notin \{-(\ga_1+2\ga_2)+\gd, -(\ga_1+3\ga_2)+\gd, -(2\ga_1+3\ga_2)+\gd, -(\ga_1+3\ga_2)+2\gd, -(2\ga_1+3\ga_2)+2\gd\}
\]
and $t\fh[t]v_M=0$ are easily proved from the definition.
For example when $\ggg = -(\ga_1+2\ga_2)+2\gd$, $x_{\ggg}v_M = 0$ follows since $[x_{-(\ga_1+\ga_2)+\gd}, x_{-\ga_2+\gd}]v_M = 0$.

For computational convenience, we assume from now on that the root vectors are normalized so that
\begin{align*} [e_{\ga_2}, f_{\ga_1+3\ga_2}] = f_{\ga_1+2\ga_2}, \ [e_{\ga_2}, f_{\ga_1+2\ga_2}] &= f_{\ga_1+\ga_2}, \ 
  [e_{\ga_2}, f_{\ga_1+\ga_2}] = f_{\ga_1},\\
   [f_{\ga_1+\ga_2}, f_{\ga_1+2\ga_2}] &= 6f_{2\ga_1+3\ga_2}.
\end{align*}
For an element $X$ in an algebra and $p \in \Z_+$ denote by $X^{(p)}$ the divided power $X^p/p!$,
and set $X^{(p)} =0$ if $p <0$.


\begin{Lem}\label{Lemma:key}
 {\normalfont(i)} For $q \in \Z_+$, we have
 \[ e_{\ga_2} f_{\ga_1+2\ga_2}^{(q)} \equiv 3 f_{2\ga_1+3\ga_2}f_{\ga_1+2\ga_2}^{(q-2)} \ \ \ 
    \mathrm{mod} \ U(\fg)\big(\C e_{\ga_2} \oplus \C f_{\ga_1} \oplus 
    \C f_{\ga_1+\ga_2}).
 \]
 {\normalfont (ii)} For $p,q\in \Z_+$, we have
 \[ e_{\ga_2}^{(p)}f_{\ga_1+3\ga_2}^{(q)} \equiv \sum_{i} f_{2\ga_1+3\ga_2}^{(i)}
    f_{\ga_1+3\ga_2}^{(q-p+i)}f_{\ga_1+2\ga_2}^{(p-3i)} \ \ \mathrm{mod} \ U(\fg)\big(\C e_{\ga_2} \oplus \C f_{\ga_1}
    \oplus \C f_{\ga_1+\ga_2}),
 \]
 where $i$ runs over the set of integers such that $\max\{0, p-q\} \le i \le p/3$.
\end{Lem}

\begin{proof}
 We have
 \begin{align*}
  e_{\ga_2} f_{\ga_1+2\ga_2}^{(q)} &\equiv \frac{1}{q!}\sum_{i=1}^q f_{\ga_1+2\ga_2}^{i-1}f_{\ga_1+\ga_2}f_{\ga_1+2\ga_2}^{q-i}
   \equiv \frac{1}{q!}\sum_{i=1}^q 6(q-i)f_{2\ga_1+3\ga_2}f_{\ga_1+2\ga_2}^{q-2}\\ &= \frac{1}{q!}\cdot 
   3q(q-1)f_{2\ga_1+3\ga_2} f_{\ga_1+2\ga_2}^{q-2} = 3f_{2\ga_1+3\ga_2}f_{\ga_1+2\ga_2}^{(q-2)},
 \end{align*}
 and the assertion (i) holds. 
 The assertion (ii) with $p = 1$ is immediate. 
 Then we have by induction and (i) that
 \begin{align*}
  &(p+1)e_{\ga_2}^{(p+1)}f_{\ga_1+3\ga_2}^{(q)} \equiv e_{\ga_2} \sum_{i} 
    f_{2\ga_1+3\ga_2}^{(i)}f_{\ga_1+3\ga_2}^{(q-p+i)}f_{\ga_1+2\ga_2}^{(p-3i)}\\
  \equiv&  \sum_{i} f_{2\ga_1+3\ga_2}^{(i)} \Big(f_{\ga_1+3\ga_2}^{(q-p+i-1)}
     f_{\ga_1+2\ga_2}f_{\ga_1+2\ga_2}^{(p-3i)} + 3f_{2\ga_1+3\ga_2}f_{\ga_1+3\ga_2}^{(q-p+i)}
     f_{\ga_1+2\ga_2}^{(p-3i-2)}\Big)\\
  =& \sum_{i} (p-3i+1)f_{2\ga_1+3\ga_2}^{(i)}f_{\ga_1+3\ga_2}^{(q-p+i-1)}
   f_{\ga_1+2\ga_2}^{(p-3i+1)}
   + \sum_{i} 3(i+1)f_{2\ga_1+3\ga_2}^{(i+1)}f_{\ga_1+3\ga_2}^{(q-p+i)}f_{\ga_1+2\ga_2}^{(p-3i-2)}\\
   =& (p+1)\sum_{i} f_{2\ga_1+3\ga_2}^{(i)} f_{\ga_1+3\ga_2}^{(q-p+i-1)}f_{\ga_1+2\ga_2}^{(p-3i+1)}.
 \end{align*}
 Hence the assertion (ii) holds.
\end{proof}

By Lemma \ref{Lemma:key} (ii), we also see that
\begin{align}\label{eq:Lemma(ii)}
 e_{\ga_2}^{(p)}(f_{\ga_1+3\ga_2}t)^{(q)} \equiv \sum_{i=\max\{0,p-q\}}^{\lfloor p/3 \rfloor} 
 (f_{2\ga_1+3\ga_2}t^2)^{(i)}
 (f_{\ga_1+3\ga_2}&t)^{(q-p+i)}(f_{\ga_1+2\ga_2}t)^{(p-3i)}\\ \ \mathrm{mod}& \ U(\fg)\big(\C e_{\ga_2} \oplus \C f_{\ga_1}t
 \oplus \C f_{\ga_1+\ga_2}t).\nonumber
\end{align}

\begin{Lem}\label{Lemma:rel1}
 The relations $(f_{\ga_1+3\ga_2}t)^{2r+1}v_M=0$ and $(f_{2{\ga_1}+3{\ga_2}}t)^{k+2r+1}v_M=0$ hold.
\end{Lem}

\begin{proof}
 We have 
 \[ \big\langle \ga_2^{\vee}, \wt\big((f_{\ga_1+3\ga_2}t)^{2r+1}v_M\big)\big\rangle = \langle \ga_2^{\vee}, \gl - (2r+1)
    (\ga_1+3\ga_2)\rangle =-(3r+3).
 \]
 On the other hand, it follows from (\ref{eq:Lemma(ii)}) that 
 \[ e_{\ga_2}^{3r+3}(f_{{\ga_1}+3{\ga_2}}t)^{2r+1}v_M=0,
 \]
 and hence we have $(f_{{\ga_1}+3{\ga_2}}t)^{2r+1}v_M=0$ since $M(\gl)$ is an integrable $\fg$-module.
 Now it is an elementary fact that this relation and $f_{\ga_1}^{k+1}v_M=0$ imply $(f_{2{\ga_1}+3{\ga_2}}t)^{k+2r+1}v_M=0$
 (for example, see \cite[Lemma 4.5]{MR2855081}).
\end{proof}

\begin{Lem}
 The relations $(f_{2{\ga_1}+3{\ga_2}}t^2)^{r+1}v_M=0$ and $(f_{{\ga_1}+3{\ga_2}}t^2)^{r+1}v_M=0$ hold.
\end{Lem}

\begin{proof}
 By Lemma \ref{Lemma:rel1} and (\ref{eq:Lemma(ii)}), we have
 \[ 0 = e_{{\ga_2}}^{(3r+3)}(f_{{\ga_1}+3{\ga_2}}t)^{(2r+2)}v_M = (f_{2{\ga_1}+3{\ga_2}}t^2)^{(r+1)}v_M,
 \]
 and hence $(f_{2{\ga_1}+3{\ga_2}}t^2)^{r+1}v_M=0$ follows. From this we see that
 \[ 0= e_{{\ga_1}}^{r+1}(f_{2{\ga_1}+3{\ga_2}}t^2)^{r+1}v_M = c(f_{{\ga_1}+3{\ga_2}}t^2)^{r+1}v_M
 \]
 with some nonzero $c$. Hence $(f_{{\ga_1}+3{\ga_2}}t^2)^{r+1}v_M=0$ also holds.
\end{proof}

\begin{Lem}\label{Lemma:rel2}
 The relation $(f_{{\ga_1}+2{\ga_2}}t)^{3r+1}v_M =0$ holds.
\end{Lem}

\begin{proof}
 By Lemma \ref{Lemma:rel1} and (\ref{eq:Lemma(ii)}), we have for $p \geq 2r+1$ that
 \begin{align*} 
   0 = e_{\ga_2}^{(p)}(f_{{\ga_1}+3{\ga_2}}t)^{(p)}v_M
     = \sum_{i=0}^{\lfloor p/3\rfloor} \frac{1}{(p-3i)!}(f_{2{\ga_1}+3{\ga_2}}t^2)^{(i)}
    (f_{{\ga_1}+3{\ga_2}}t)^{(i)}(f_{{\ga_1}+2{\ga_2}}t)^{p-3i}v_M.
 \end{align*}
 When $2r+1 \le p \le 3r+1$, by multiplying $(f_{{\ga_1}+2{\ga_2}}t)^{3r+1-p}$ to this
 we obtain $r$ linear relations
 \[ \sum_{i=0}^{\lfloor p/3\rfloor} \frac{1}{(p-3i)!} (f_{2{\ga_1}+3{\ga_2}}t^2)^{(i)}
    (f_{{\ga_1}+3{\ga_2}}t)^{(i)}(f_{{\ga_1}+2{\ga_2}}t)^{3r+1-3i}v_M=0.
 \]
 Hence in order to prove $(f_{{\ga_1}+2{\ga_2}}t)^{3r+1}v_M =0$, it is enough to show that the matrix 
 $A=(a_{ij})_{0\leq i,j\leq r}$ with
 \[ a_{ij} = \begin{cases}
                \frac{1}{(3r+1-3i-j)!} & \text{if} \ 3r+1-3i-j \geq 0,\\
                0                 & \text{otherwise}
             \end{cases}
 \]
 is invertible.
 Assume that $v_0,v_1,\ldots,v_r$ satisfy $\sum a_{ij} v_j = 0$ for all $i$, and consider the polynomial 
 \begin{align*}
  f(x) = &\frac{v_0}{(3r+1)!}x^{3r+1} + \frac{v_1}{(3r-2)!}x^{3r-2}
     + \cdots + \frac{v_i}{(3r+1-3i)!}x^{3r+1-3i} + \cdots + \frac{v_{r}}{1!}x.
 \end{align*}
 Then $\frac{d^jf}{dx^j}(1)=0$ holds for all $0\leq j \leq r$, 
 which implies that $f(x)$ is divisible by $(x-1)^{r+1}$.
 Since $f(\zeta x) =\zeta f(x)$ holds where $\zeta$ is a third primitive root of unity, 
 we see that $f(x)$ is divisible by $(x^3-1)^{r+1}$.
 By the degree consideration we have $f(x) =0$, and the proof is complete.
\end{proof}

Now the following lemma completes the proof of Proposition \ref{Prop:annihilation}.
\begin{Lem}
 The relation $f_{{\ga_1}+3{\ga_2}}t^2(f_{{\ga_1}+2{\ga_2}}t)^{3r-2}v_M =0$ holds 
 when $r\neq 0$.
\end{Lem}

\begin{proof}
 Let $p \geq 2r-1$.
 By Lemma \ref{Lemma:rel1}, we have
 \begin{align}\label{eq:Garland}
   0&=e_{{\ga_1}+3{\ga_2}}(f_{{\ga_1}+3{\ga_2}}t)^{(p+2)}v_M = \frac{1}{(p+2)!}\sum_{i=0}^{p+1} (f_{{\ga_1}+3{\ga_2}}t)^{p-i+1}
    (\ga_1+3\ga_2)^{\vee}t (f_{\ga_1+3\ga_2}t)^{i}v_M\nonumber\\ 
    & = \frac{1}{(p+2)!}\sum_{i=0}^{p+1} -2i
    (f_{{\ga_1}+3{\ga_2}}t)^{p} f_{{\ga_1}+3{\ga_2}}t^2v_M= -(f_{{\ga_1}+3{\ga_2}}t)^{(p)}f_{{\ga_1}+3{\ga_2}}t^2v_M.
 \end{align}
 We easily see that all the elements $e_{\ga_2}, f_{\ga_1}t, f_{\ga_1+\ga_2}t$ annihilate the vector 
 $f_{{\ga_1}+3{\ga_2}}t^2v_M$, and hence we have from (\ref{eq:Lemma(ii)}) and (\ref{eq:Garland}) that
 \begin{align*} 
  0&=e_{{\ga_2}}^{(p)}(f_{{\ga_1}+3{\ga_2}}t)^{(p)}f_{{\ga_1}+3{\ga_2}}t^2v_M \\
   &= \sum_{i=0}^{\lfloor p/3\rfloor}\frac{1}{(p-3i)!} f_{{\ga_1}+3{\ga_2}}t^2 (f_{2{\ga_1} + 3{\ga_2}}t^2)^{(i)}
     (f_{{\ga_1}+3{\ga_2}}t)^{(i)}(f_{{\ga_1}+2{\ga_2}}t)^{p-3i}.
 \end{align*}
 Now the lemma is proved by a similar argument as in the proof of Lemma \ref{Lemma:rel2}.
\end{proof}

\section{Proof of Theorem \ref{polyhedral multiplicity formula}}\label{Section:Polyhedral}

\subsection{A basis of the space of highest weight vectors}

For $\bm{a} = (a_1,a_2,a_3,a_4,a_5)\in \Z_+^5$, set
\[ \bf_{\ba}=(f_{2\ga_1+3\ga_2}t^2)^{(a_5)}(f_{\ga_1+3\ga_2}t^2)^{(a_4)}(f_{\ga_1+3\ga_2}t)^{(a_3)}
   (f_{\ga_1+2\ga_2}t)^{(a_2)}(f_{2\ga_1+3\ga_2}t)^{(a_1)},
\]
and  
\begin{align*}
  \wt(\ba) &= (2a_1+a_2+a_3+a_4+2a_5)\ga_1 + (3a_1+2a_2+3a_3+3a_4+3a_5)\ga_2\\
           &=(a_1 -a_3-a_4+a_5)\go_1 + (a_2 + 3a_3+3a_4)\go_2 \in Q_+.
\end{align*}
Note that $\wt(\bf_{\ba}) = - \wt(\ba)$.
In this section, we denote by $v$ a highest weight vector of $L(m)$.
Since $L(m) \cong M(\gl)$, we easily see from Proposition \ref{Prop:annihilation} and the PBW theorem that 
\[ L(m) = \sum_{\ba \in \Z_+^5} U(\fg)\bf_{\ba}v.
\]

Let $\ga \in Q_+$, and set $L(m)_{>\gl-\ga} = \bigoplus_{\mu > \gl-\ga} L(m)_\mu$.
The $\fg$-submodule $U(\fg)L(m)_{> \gl-\ga}$ of $L(m)$ coincides with the sum of simple $\fg$-submodules whose 
highest weights are larger than $\gl -\ga$.
Hence we see that the multiplicity of $V(\gl-\ga)$ in $L(m)$ is equal to the dimension of the weight space of the quotient 
$\fg$-module $L(m)\Big/U(\fg)L(m)_{>\gl-\ga}$
with weight $\gl-\ga$, that is
\[ \Big[L(m) : V(\gl-\ga)\Big] = \dim \Big(L(m)\Big/U(\fg)L(m)_{>\gl-\ga}\Big)_{\gl-\ga}.
\]
Therefore, in order to prove Theorem \ref{polyhedral multiplicity formula} 
it suffices to show the following proposition, which is proved in the next subsections.

\begin{Prop}\label{Proposition:basis}
 For every $\ga \in Q_+$, the projection images of $\{\bf_{\ba}v \mid \ba \in S_\gl, \ \wt(\ba) = \ga\}$ form a basis of
 $\Big(L(m)\Big/U(\fg)L(m)_{>\gl-\ga}\Big)_{\gl-\ga}$.
\end{Prop}

\subsection{The space is spanned by the vectors}

For $\ga \in Q_+$, set 
\[ \Z_+^5[\ga] = \{ \ba \in \Z_+^5 \mid \wt(\ba) = \ga\}, \ \ \ S_\gl[\ga] = S_\gl \cap \Z_+^5[\ga].
\]
In this subsection, we shall show the following.

\begin{Lem}\label{Lemma:span}
 For every $\ga \in Q_+$, the projection images of $\{\bf_{\ba}v\mid \ba \in S_\gl[\ga]\}$
 span the space $\Big(L(m)\Big/U(\fg)L(m)_{>\gl-\ga}\Big)_{\gl-\ga}$.
\end{Lem}

We denote by $\le$ the lexicographic order on $\Z_+^5$, that is, 
$(a_1,\ldots,a_5) < (b_1,\ldots,b_5)$ if and only if there exists $i$ such that 
$a_j = b_j$ for $j < i$ and $a_i < b_i$. 
Fix $\ga \in Q_+$.
Following \cite[Subsection 3.5]{MR2372556}, we define a finite sequence $\br_1,\ldots,\br_t$ of elements of $\Z_+^5[\ga]$
inductively as follows. 
Set $\br_1$ to be the least element (with respect to the lexicographic order) of $\Z_+^5[\ga]$ such that $\bf_{\br_1}v \notin 
U(\fg)L(m)_{>\gl-\ga}$.
Assume that $\br_1,\ldots,\br_{p}$ are defined.
We set $\br_{p+1}$ to be the least element of $\Z_+^5[\ga]$ such that 
\[ \bf_{\br_{p+1}}v \notin \sum_{i=1}^{p}\C\bf_{\br_i}v + U(\fg)L(m)_{>\gl-\ga}
\]
if such an element exists, and otherwise we set $t = p$.

Set $K[\ga] = \{\br_1,\ldots,\br_t\}$.
By the definition the projection images of $\{\bf_{\ba}v \mid \ba \in K[\ga]\}$
 span $\Big(L(m)\Big/U(\fg)L(m)_{>\gl-\ga}\Big)_{\gl-\ga}$,
and every $\br \in K[\ga]$ satisfies that
\begin{equation}\label{eq:condition_of_K}
 \bf_{\br}v \notin \sum_{\begin{smallmatrix} \ba \in \Z_+^5[\ga], \\ \ba < \br\end{smallmatrix}} 
 \C \bf_{\ba}v + U(\fg)L(m)_{>\gl-\ga}.
\end{equation}
It is enough to show that every $\br = (r_1,\ldots,r_5) \in K[\ga]$ satisfies
\[ r_1 \leq k, \ \ \ r_1 - r_3 + r_5 \leq k, \ \ \ 2r_2 + 3r_3 + 3r_4 \leq l,\ \ \ 2r_2+3r_4+3r_5 \leq l,
\]
since this implies $K[\ga] \subseteq S_\gl[\ga]$.

Fix $\br = (r_1,\ldots,r_5) \in K[\ga]$, and first assume that $r_1 > k$.
The Lie subalgebra of $\fg[t]$ spanned by $f_{\ga_1}$, $f_{\ga_1+3\ga_2}t$, and $f_{2\ga_1+3\ga_2}t$ is isomorphic to
the $3$-dimensional Heisenberg algebra. 
Then \cite[Lemma 1.5]{MR2372556} and $f_{\ga_1}^{k+1}v=0$ imply that
\[ (f_{\ga_1+3\ga_2}t)^{r_3}(f_{2\ga_1+3\ga_2}t)^{r_1}v\in 
   \sum_{0<p, 0\leq q, 0\leq s \leq k} \C f_{\ga_1}^p (f_{\ga_1+3\ga_2}t)^q (f_{2\ga_1+3\ga_2}t)^sv.   
\]
From this we easily see that 
\[ \bf_{\br}v \in \sum_{\begin{smallmatrix}\ba \in \Z_+^5[\ga], \\ \ba<\br
     \end{smallmatrix}} \C\bf_{\ba} v + U(\fg)L(m)_{>\gl-\ga},
\]
which contradicts (\ref{eq:condition_of_K}).

Next assume that $r_1 - r_3 + r_5 > k$.
Let $\be_i$ ($1\le i \le 5$) denote the standard basis of $\Z^5$,
and set $\bs = \br -r_4\be_4 + r_4\be_5$.
We easily see that 
\begin{equation}\label{eq:containment}
 e_{\ga_1}^{r_4} \bf_{\bs}v \in \C^\times\bf_{\br}v + \sum_{\begin{smallmatrix} \ba \in \Z_+^5[\ga], \\ 
 \ba < \br\end{smallmatrix}} \C \bf_{\ba}v.
\end{equation} 
Note that
\[ \wt(\bf_{\bs}v) = \gl - \ga - r_4\ga_1 = (k-r_1 + r_3 - r_4-r_5)\go_1 + (l-r_2-3r_3)\go_2,
\]
and hence we have 
\[ s_1\wt(\bf_{\bs}v) = \gl -\ga + (r_1 - r_3 + r_5-k)\ga_1 > \gl -\ga,
\]
which implies $\bf_{\bs}v \in U(\fg)L(m)_{>\gl - \ga}$.
Then this and (\ref{eq:containment}) contradict (\ref{eq:condition_of_K}).

The inequality $2r_2+3r_3+3r_4 \leq l$ is proved in the same way as in \cite[Subsection 3.5]{MR2372556}.

Finally assume that $2r_2+3r_4+3r_5 > l$.
Then $r_5 > r_3$ follows, since otherwise we have
$2r_2 + 3r_4 + 3r_5 \leq 2r_2 + 3r_3 + 3r_4 \leq l$.
Set 
\[ \bs_j = (r_1,0,r_2+r_3+2r_5 -2j, r_4,j) \ \ \ \text{for} \ 0\leq j \leq r_3.
\]
We have
\[ \wt(\bf_{\bs_j}v) = \gl - \ga - (r_2 + 3r_5-3j)\ga_2, \ \ \ \langle \wt(\bf_{\bs_j}v), \ga_2^\vee\rangle
   = l - 3r_2 -3r_3 -3r_4 - 6r_5 +6j.
\]
Then by a similar argument as in the proof of $r_1 - r_3 + r_5 \leq k$, we can show that
\begin{equation}\label{eq:sj}
 \bf_{\bs_j}v \in U(\fg)L(m)_{>\gl-\ga} \ \ \ \text{for all} \ 0\leq j \leq r_3.
\end{equation}
It follows from (\ref{eq:Lemma(ii)}) that 
\begin{align*}
 e_{\ga_2}^{(r_2+3r_5-3j)}\bf_{\bs_j}v&= \sum_{i=\max\{0,r_5-r_3-j\}}^{r_5-j+\lfloor r_2/3\rfloor}\begin{pmatrix} i+j \\ j 
 \end{pmatrix}\bf_{(r_1,r_2+3r_5-3i-3j,r_3-r_5+i+j,r_4,i+j)}v\\
 &= \sum_{i=-\lfloor r_2/3\rfloor}^{\min\{r_5-j,r_3\}} \begin{pmatrix} r_5-i \\ j \end{pmatrix}
 \bf_{(r_1,r_2+3i,r_3-i,r_4,r_5-i)}v\\
 &\in \sum_{i=0}^{\min\{r_5-j,r_3\}} \begin{pmatrix} r_5-i \\ j \end{pmatrix} \bf_{\br+i(3\be_2-\be_3-\be_5)}v
 + \sum_{\begin{smallmatrix} \ba \in \Z_+^5[\ga], \\ \ba < \br\end{smallmatrix}} 
 \C \bf_{\ba}v,
\end{align*}
and then by (\ref{eq:sj}) we have for every $0\le j \le r_3$ that
\[ \sum_{i = 0}^{\min\{r_5-j,r_3\}} \begin{pmatrix} r_5-i \\ j \end{pmatrix} \bf_{\br+i(3\be_2-\be_3-\be_5)}v
   \in \sum_{\begin{smallmatrix} \ba \in \Z_+^5[\ga], \\ \ba < \br\end{smallmatrix}} 
   \C \bf_{\ba}v + U(\fg)L(m)_{>\gl-\ga}.
\]
From this we can show that
\[ \bf_{\br}v \in \sum_{\begin{smallmatrix} \ba \in \Z_+^5[\ga], \\ \ba < \br\end{smallmatrix}} 
   \C \bf_{\ba}v + U(\fg)L(m)_{>\gl-\ga}
\]
by a similar argument as in Lemma \ref{Lemma:rel2}, in which we use a polynomial
\[ f(x) = v_0x^{r_5} + v_1x^{r_5-1} + \cdots +v_{r_3}x^{r_5-r_3}
\]
instead. 
Now this contradicts (\ref{eq:condition_of_K}).

\subsection{Linearly independence}

Proposition \ref{Proposition:basis} is proved from the following lemma, together with Lemma \ref{Lemma:span}.

\begin{Lem}\label{Lem:linearly_independence}
 For every $\ga \in Q_+$, the images of $\{\bf_{\ba}v\mid \ba \in S_\gl[\ga]\}$ under the canonical projection
 $L(m) \twoheadrightarrow L(m)/U(\fg)L(m)_{>\gl-\ga}$ are linearly independent.
\end{Lem}

Fix $\ga \in Q_+$.
Let $\ol{L(m)} = L(m)/U(\fg)L(m)_{>\gl-\ga}$, and $\pr$ denote the canonical projection 
$L(m) \twoheadrightarrow \ol{L(m)}$.
We shall show the lemma by the induction on $k$.
The case $k=0$ is proved in \cite{MR2372556}.

Assume that $k>0$, and a sequence $\{c_{\ba}\}_{\ba \in S_\gl[\ga]}$ of complex numbers satisfies
\begin{equation}\label{eq:sum4}
 \sum_{\ba \in S_\gl[\ga]}c_{\ba}\pr(\bf_{\ba}v) = 0.
\end{equation}
First we shall show that
\begin{equation}\label{eq:sum2}
 c_{\ba} = 0 \ \text{for all} \ \ba \in S_\gl[\ga] \ \text{such that} \ a_1>0.
\end{equation}
Let $L_1$ and $L_2$ be the graded limits of minimal affinizations of $V_q(\go_1)$ and $V_q(\gl-\go_1)$ respectively, 
and $v_1,v_2$ be respective highest weight vectors.
Set $\gl_2=\gl-\go_1$.
It follows that
\[ L(m) \cong T(\gl) \hookrightarrow T(k\go_1) \otimes T(l\go_2) \hookrightarrow T(\go_1)\otimes T\big((k-1)\go_1\big) 
   \otimes T(l\go_2),
\]
and from this we see that $L(m) \cong U(\fg[t])(v_1 \otimes v_2) \subseteq L_1 \otimes L_2$.
It is known that 
\[ L_1 =U(\fg)v_1 \oplus U(\fg)\bf_{\be_1}v_1 \cong V(\go_1) \oplus V(0)
\]
as a $\fg$-module, and $\bf_{\ba}v_1 = 0$ if $\ba \notin \{0, \be_1\}$.

Let $\pr^1\colon L_1 \twoheadrightarrow V(0)$ be the projection with respect to the $\fg$-module decomposition,
and $\pr^2_{\gl-\ga}\colon L_2 \twoheadrightarrow L_2/U(\fg)(L_2)_{>\gl-\ga}$ the canonical projection. 
Since 
\[ (L_1\otimes L_2)_{>\gl-\ga} = \bigoplus_{\mu \in P} (L_1)_\mu \otimes (L_2)_{>\gl-\ga-\mu} \subseteq 
   V(0) \otimes (L_2)_{>\gl-\ga} \oplus V(\go_1)\otimes L_2,
\]
we have
\begin{align*}
 U(\fg)(L_1 \otimes L_2)_{> \gl-\ga} \subseteq V(0) \otimes U(\fg)(L_2)_{>\gl-\ga} \oplus V(\go_1) \otimes L_2.
\end{align*}
Hence the composition 
\[ \kappa\colon L(m) \hookrightarrow L_1 \otimes L_2 \stackrel{\pr^1\otimes \pr_{\gl-\ga}^2}{\twoheadrightarrow}
   V(0) \otimes \Big(L_2/U(\fg)(L_2)_{>\gl-\ga}\Big)
   \cong L_2/U(\fg)(L_2)_{>\gl-\ga}
\]
induces a $\fg$-module homomorphism $\ol{\kappa}\colon\ol{L(m)} \to L_2/U(\fg)(L_2)_{>\gl-\ga}$.
It is easily seen for $\ba = (a_1,\ldots,a_5)$ that 
\begin{align}\label{eq:sum3}
 \bf_{\ba}(v_1 \otimes v_2)= \begin{cases} v_1 \otimes \bf_{\ba}v_2 + \bf_{\be_1}v_1 \otimes \bf_{\ba-\be_1}v_2 & \text{if}
 \ a_1 > 0,\\
 v_1 \otimes \bf_{\ba}v_2 & \text{otherwise}.
 \end{cases}
\end{align}
Hence we see from the definition of $\kappa$ that (\ref{eq:sum4}) yields 
\begin{equation*}
 0=\ol{\kappa}\Big(\sum_{\ba \in S_\gl[\ga]} c_{\ba} \pr(\bf_{\ba}v)\Big) 
  = \sum_{\ba \in S_{\gl}[\ga]} c_{\ba} \kappa (\bf_{\ba}v)
 = \sum_{\ba \in S_\gl[\ga]:\, a_1 >0} c_{\ba} \pr^2_{\gl-\ga}(\bf_{\ba-\be_1}v_2).
\end{equation*}
Since $\gl-\ga = \gl_2-(\ga-\go_1)$ and $\{\ba - \be_1\mid \ba \in S_\gl[\ga], a_1 >0\} \subseteq S_{\gl_2}[\ga-\go_1]$, 
(\ref{eq:sum2}) follows from the induction hypothesis, as required.

Set 
\[ S^0_\gl[\ga] = \{\ba \in S_\gl[\ga]\mid a_1=0\} \ \ \text{and} \ \ S^{0,k}_\gl[\ga] = \{\ba \in S_\gl[\ga]\mid a_1=0,
   -a_3+a_5=k\} \subseteq S_\gl^0[\ga].
\]
It is easily checked that
\begin{equation}\label{eq:sqcup}
 S^0_\gl[\ga] = S^0_{\gl_2}[\ga] \sqcup S^{0,k}_\gl[\ga].
\end{equation}
Next we would like to prove that
\begin{equation}\label{eq:0forS}
 c_{\ba} = 0 \ \ \text{for all} \ \ba \in S^0_{\gl_2}[\ga],
\end{equation}
and in order to do that we will first prove that
\begin{equation}\label{eq:claim}
 \bf_{\ba}v_2 \in \C^\times f_{\ga_1}\bf_{\ba+(\be_4-\be_5)}v_2 + U(\fg)(L_2)_{>\gl_2-(\ga -\ga_1)} \ \ \text{if} \ 
 \ba \in S^{0,k}_\gl[\ga].
\end{equation}
Assume that $\br =(0,r_2,r_3,r_4,r_3+k) \in S^{0,k}_\gl[\ga]$. 
We see by a direct calculation that
\begin{equation}\label{eq:sl2}
 e_{\ga_1}^{r_4}\bf_{\br+r_4(\be_5-\be_4)}v_2 \in \C^\times \bf_{\br}v_2 \ \ \text{and} \ \ e_{\ga_1}^{r_4+1}
  \bf_{\br+r_4(\be_5-\be_4)}v_2 \in \C^\times \bf_{\br+(\be_4-\be_5)}v_2.
\end{equation} 
Since 
\[ \wt(f_{\ga_1}\bf_{\br+r_4(\be_5-\be_4)}v_2) = -(r_4+3)\go_1 + (l-r_2 -3r_3+3)\go_2,
\]
it follows that
\[ s_1 \wt(f_{\ga_1}\bf_{\br+r_4(\be_5-\be_4)}v_2) = \wt(\bf_{\br}v_2) + 2\ga_1 > \gl_2-(\ga-\ga_1),
\]
which implies $f_{\ga_1}\bf_{\br+r_4(\be_5-\be_4)}v_2 \in U(\fg)(L_2)_{>\gl_2-(\ga-\ga_1)}.$
Hence it follows that
\begin{align*}
 f_{\ga_1}e_{\ga_1}^{r_4+1}\bf_{\br+r_4(\be_5-\be_4)}v_2 &= (e_{\ga_1}^{r_4+1}f_{\ga_1} + [f_{\ga_1},e_{\ga_1}^{r_4+1}])
 \bf_{\br+r_4(\be_5-\be_4)}v_2\\
 &\in \C^\times e_{\ga_1}^{r_4}\bf_{\br+r_4(\be_5-\be_4)}v_2 + U(\fg)(L_2)_{>\gl_2-(\ga-\ga_1)},
\end{align*}
which together with (\ref{eq:sl2}) imply (\ref{eq:claim}).
Let $\pr^2_{\gl_2-\ga}\colon L_2 \twoheadrightarrow L_2/U(\fg)(L_2)_{>\gl_2-\ga}$ be the canonical projection.
Since $U(\fg)(L_1 \otimes L_2)_{>\gl-\ga} \subseteq L_1\otimes U(\fg)(L_2)_{>\gl_2-\ga}$,
the composition 
\[ L(m) \hookrightarrow L_1\otimes L_2 \twoheadrightarrow L_1 \otimes \big(L_2/U(\fg)(L_2)_{>\gl_2-\ga}\big)
\]
induces a $\fg$-module homomorphism $\ol{L(m)}\to L_1 \otimes \big(L_2/U(\fg)(L_2)_{>\gl_2-\ga}\big)$.
We see from (\ref{eq:claim}) that $\pr^2_{\gl_2-\ga}(\bf_{\ba}v_2)=0$ if $\ba \in S^{0,k}_\gl[\ga]$, 
and then (\ref{eq:sum2}), (\ref{eq:sum3}), (\ref{eq:sqcup}) and the induced homomorphism yield 
\[ v_1 \otimes \Big(\sum_{\ba \in S^0_{\gl_2}[\ga]} c_{\ba}\pr^2_{\gl_2-\ga}(\bf_{\ba} v_2)\Big) =0.
\]
By the induction hypothesis this implies (\ref{eq:0forS}), as required.

We have
\begin{equation}\label{eq:sum5}
 \sum_{\ba \in S^{0,k}_\gl[\ga]} c_{\ba}\pr(\bf_{\ba}v) = 0
\end{equation}
by (\ref{eq:sum4}), (\ref{eq:sum2}) and (\ref{eq:0forS}), and
it remains to show that $c_{\ba} = 0$ for $\ba \in S^{0,k}_\gl[\ga]$.
Fix $\br=(r_1,\ldots,r_5) \in S^{0,k}_\gl[\ga]$, and set $\bs = \br + \be_4-\be_5$.
We define a $\fg$-submodule $L_2'$ of $L_2$ by
\[ L_2' = \sum_{\begin{smallmatrix} \ba \in S_{\gl_2} \\ \wt(\ba) < \ga,\, \ba \neq \bs \end{smallmatrix}} 
   U(\fg)\bf_{\ba}v_2.
\]
We have $(L_2)_{>\gl_2-\ga} \subseteq \C\bf_{\bs}v_2 + L_2'$ by Lemma \ref{Lemma:span},
and from this we see that
\begin{align*}
 (L_1\otimes L_2)_{>\gl-\ga} &= \C v_1 \otimes (L_2)_{>\gl_2-\ga} \oplus \bigoplus_{\gb > 0} (L_1)_{\go_1-\gb}
 \otimes (L_2)_{>\gl_2-\ga+\gb} \\ 
 &\subseteq \C v_1 \otimes \bf_{\bs}v_2 + L_1 \otimes L_2',
\end{align*}
which implies $U(\fg)(L_1 \otimes L_2)_{>\gl-\ga} \subseteq U(\fg)(v_1 \otimes \bf_{\bs}v_2) + L_1 \otimes L_2'$.
Hence the composition 
\[ \rho\colon L(m) \hookrightarrow L_1 \otimes L_2 \twoheadrightarrow (L_1 \otimes L_2)
   \Big/\big(U(\fg)(v_1 \otimes \bf_{\bs}v_2) + L_1 \otimes L_2'\big)
\]
induces a $\fg$-module homomorphism
\[ \ol{\rho}\colon \ol{L(m)} \to (L_1 \otimes L_2)\Big/\big(U(\fg)(v_1 \otimes \bf_{\bs}v_2) + L_1 \otimes L_2'\big).
\]
If $\ba \in S_{\gl}^{0,k}[\ga] \setminus \{\br\}$, then we have $\ba + \be_4-\be_5 \in S_{\gl_2}[\ga-\ga_1]\setminus \{\bs\}$ 
and hence it follows by (\ref{eq:claim}) that
\[ \bf_{\ba}(v_1 \otimes v_2) = v_1 \otimes \bf_{\ba}v_2 \in L_1 \otimes L_2'.
\]
Hence we have from (\ref{eq:sum5}) that
\[ 0 = \ol{\rho}\Big(\sum_{\ba \in S^{0,k}_\gl[\ga]} c_{\ba}\pr(\bf_{\ba}v)\Big) =\sum_{\ba \in S^{0,k}_\gl[\ga]} c_{\ba} 
       \rho(\bf_{\ba}v) = c_{\br}\rho(\bf_{\br}v).
\]
Assume that $c_{\br} \neq 0$, which implies $\rho(\bf_{\br}v) = 0$.
Let $\pr_2'$ denote the canonical projection $L_2 \to L_2/L_2'$.
We easily see that $\rho(\bf_{\br}v) = 0$ is equivalent to
\begin{equation}\label{eq:containment2}
 v_1 \otimes \pr_2'(\bf_{\br}v_2) \in U(\fg)\big(v_1 \otimes \pr_2'(\bf_{\bs}v_2)\big).
\end{equation}
Note that $\pr_2'(\bf_{\bs}v_2) \neq 0$ by the induction hypothesis,
and this also implies $\pr_2'(\bf_{\br}v_2) \neq 0$ since $e_{\ga_1}\pr_2'(\bf_{\br}v_2) \in \C^\times \pr_2'(\bf_{\bs}v_2)$ 
by (\ref{eq:sl2}).
Since 
\[ \fn_+\big(v_1 \otimes \pr_2'(f_{\bs}v_2)\big)=0 \ \text{and} \
   \wt\big(v_1 \otimes \pr_2'(\bf_{\br}v_2)\big) = \wt\big(v_1 \otimes \pr_2'(\bf_{\bs}v_2)\big)-\ga_1,
\]
(\ref{eq:containment2}) implies
\[ v_1 \otimes \pr_2'(\bf_{\br}v_2) \in \C f_{\ga_1}\big(v_1 \otimes \pr_2'(\bf_{\bs}v_2)\big).
\]
However this contradicts $f_{\ga_1}v_1 \neq 0$.
Hence $c_{\br}=0$ holds, as required.

\section*{Acknowledgement}
The authors would like to express their gratitude to Professors Evgeny Mukhin and Masato Okado for helpful discussions. Jian-Rong Li is supported by ERC AdG Grant 247049, the PBC Fellowship Program of Israel for Outstanding Post-Doctoral Researchers from China and India, the National Natural Science Foundation of China (no.\ 11371177, 11401275), the Fundamental Research Funds for the Central Universities of China (no. lzujbky-2015-78), and Katsuyuki Naoi is supported by JSPS Grant-in-Aid for Young Scientists (B) No.\ 25800006.

\newcommand{\etalchar}[1]{$^{#1}$}
\def\cprime{$'$} \def\cprime{$'$}

\end{document}